\documentclass[12pt,a4paper, english]{article}

\usepackage[T1]{fontenc}
\usepackage{babel}		
\usepackage[utf8]{inputenc}
\usepackage{amsmath}  
\usepackage{amsthm}		
\usepackage{amsfonts}		
\usepackage{amssymb} 
\usepackage{enumerate} 
\newtheorem{teorema}{Theorem}
\newtheorem{lem}[teorema]{Lemma}

\newtheorem{cor}[teorema]{Corollary}
\newtheorem{athm}{Theorem}

\theoremstyle{definition}

\newtheorem{nota}[teorema]{Remark}

\title{On finitely generated left nilpotent braces}
\author{H. Meng\thanks{Department of Mathematics and Newtouch Center for Mathematics, Shanghai University, Shanghai 200444, PR China; \texttt{hymeng2009@shu.edu.cn}; ORCID 0000-0001-9840-5783.} \and A. Ballester-Bolinches%
\thanks{Departament de Matem\`atiques, Universitat de Val\`encia, Dr.\ Moliner, 50, 46100 Burjassot, Val\`encia, Spain; \texttt{Adolfo.Ballester@uv.es}, \texttt{Vicent.Perez-Calabuig@uv.es}; ORCID 0000-0002-2051-9075, 0000-0003-4101-8656.}
\and L. A. Kurdachenko\thanks{Department of Algebra and Geometry, Oles Honchar Dnipro National University, Dnipro 49010, Ukraine; \texttt{lkurdachenko@gmail.com}; ORCID 0000-0002-6368-7319.}\ \thanks{Part of the research of this author has been carried out in the Departament de Matem{\`a}tiques, Universitat de Val{\`e}ncia; Dr.\ Moliner, 50; 46100 Burjassot, Val\`encia, Spain.} \and  V. P\'erez-Calabuig\addtocounter{footnote}{-3}\footnotemark}

\begin{document}
\date{}
\maketitle

\begin{abstract}
A description of finitely generated left nilpotent braces of class at most two is presented in this paper. The description heavily depends on the fact that if $B$ is left nilpotent of class at most~$2$, that is $B^3 = 0$, then $B$ is right nilpotent of class at most~$3$, that is $B^{(4)} = 0$. In addition, we construct a free object in the category of finitely generated left nilpotent braces of class at most~$2$.

  

\end{abstract}

\section{Introduction}
A (\emph{left}) \emph{brace} $B$ consists of a set $B$ with two group structures, a commutative one $(B,+)$, and a multiplicative one $(B,\cdot)$, such that the following sort of (left) distributivity property holds: $a \cdot (b+c) = (a \cdot b) - a + (a \cdot c)$ for every $a,b,c\in B$.
 The defining distributivity property of $B$ yields a common identity element $0 \in B$.  Braces were defined in Rump~\cite{Rump07} as a generalisation of Jacobson radical rings $(B,+,\ast)$, where the \emph{star operation $\ast$} satisfies $a \cdot b = a + b + (a \ast b)$ for every $a,b\in B$, and it is non-associative in general. The main motivation was to introduce an algebraic structure to approach the problem of classifying involutive non-degenerate set-theoretic solutions to the Yang-Baxter equation (solutions, for short), a fundamental equation in mathematical physics and quantum group theory.

A comprehensive understanding of the algebraic structure of braces is, therefore, undoubtedly essential for advancing the classification of solutions. For instance, the concept of retraction of a solution---a method for simplifying solutions---has led naturally to the emergence of nilpotency notions in braces. Given a brace $B$, the following iterated series of $\ast$-products naturally arise:
\begin{align*}
(L)\  &B^1 = B \supseteq B^2 = B \ast B \supseteq \ldots \supseteq B^{n+1} = B \ast B^n \supseteq \ldots\\
(R)\  &B^1 = B \supseteq B^{(2)} = B \ast B \supseteq \ldots \supseteq B^{(n+1)} = B^{(n)} \ast B \supseteq \ldots
\end{align*}
Then, $B$ is said to be \emph{left (resp. right) nilpotent of class $n$} if $n$ is the smallest natural such that $B^{n+1} = 0$ (resp. $B^{(n+1)} = 0$). In particular, \textit{trivial braces} are defined to be left and right nilpotent of class~$1$.

In this context, the study of right nilpotency has proven fundamental, as it precisely characterises those solutions that can be recursively retracted until getting a trivial solution after finitely many steps---the so-called multipermutation solutions, one of the most prevalent and well-structured families of solutions of the YBE.

One of the central challenges in the algebraic theory of braces is determining the structure of finitely generated braces---if $b_1, \ldots, b_r$ are elements of a brace $B$, $\langle b_1, \ldots, b_r\rangle$ denotes the smallest subbrace containing them. Unlike in group theory, where the structure of cyclic groups is well understood, even describing one-generated braces is a problem of great difficulty. Nevertheless, notable progress has been made in case of low nilpotency classes of braces: see~\cite{BallesterEstebanKurdachenkoPerezC25-onegenerated, KurdachenkoSubbotin24,DixonKurdachenkoSubbotin25-preprint}, respectively, for right nilpotency class~$2$, left nilpotency class~$2$, and central nilpotency class~$3$. In such cases, central nilpotency---also known as strong nilpotency in braces (see~\cite{Smoktunowicz18-tams})---provides valuable insights into the structure of finitely generated braces. The \emph{centre} $\zeta(B)$ of $B$ is defined in~\cite{BonattoJedlicka23} as
\[ \zeta(B) = \{b\in B\mid a +b = b\cdot a = a\cdot b,\ \text{for all $a\in B$}\}.\]
Then, $B$ is said to be \emph{centrally nilpotent} if there exists a chain of ideals
\[ 0 = I_0 \leq I_1\leq \ldots \leq I_n = B\]
such that $I_k/I_{k-1} \leq \zeta(B/I_{k-1})$ for every $1\leq k \leq n$. It follows that $B$ is centrally nilpotent if, and only if, $B$ is left and right nilpotent (see~\cite[Corollary 2.15]{JespersVanAntwerpenVendramin23}). The notion of central nilpotency in braces closely mirrors nilpotent groups, facilitating a rigorous and detailed structural description (see~\cite{BallesterEstebanFerraraPerezCTrombetti25-central} for further details).

In this light, the following theorem in~\cite{BallesterKurdachenkoPerezC-arXiv-B3-B4} turns out to be crucial.

\begin{teorema}
\label{teo:B3-B4}
Let $B$ be left nilpotent brace of class at most~$2$, i.e. $B^3 = 0$. Then, $B$ is right nilpotent of class at most~$3$, i.e. $B^{(4)} = 0$.
\end{teorema}
This article is a natural continuation of~\cite{BallesterKurdachenkoPerezC-arXiv-B3-B4}. We show that Theorem~\ref{teo:B3-B4} lays the foundations to prove a detailed description of finitely generated left nilpotent braces of class at most~$2$.

\begin{athm}
\label{teo:A}
Let $B$ be a left nilpotent brace of class at most~$2$. Let $b_1, \ldots, b_r \in B$ such that $B = \langle b_1, \ldots, b_r\rangle$. Then, $B = L_1 + L_2 + L_3$ with
\begin{itemize}
\item $L_1 = \left\{\sum_{i=1}^r z_ib_i\mid z_i \in \mathbb{Z}, 1 \leq i \leq r\right\}$.
\item $L_2 = \left\{\sum_{i=1}^r z_{ij}b_{ij}\mid z_{ij} \in \mathbb{Z}, 1 \leq i,j \leq r\right\}$, where $b_{ij} = b_i \ast b_j$, for every $1\leq i, j\leq r$.
\item $L_3 = \left\{\sum_{i=1}^r z_{ijk}b_{ijk}\mid z_{ijk} \in \mathbb{Z}, 1 \leq i,j,k \leq r\right\}$, where $b_{ijk} = (b_i \ast b_j) \ast b_k$, for every $1\leq i, j, k\leq r$.
\end{itemize}
Therefore, the rank $r(B)$ of the additive group $(B,+)$ is less than $r + r^2 + r^3$.
\end{athm}

The rank $r(B)$ can be bounded with a lower upper bound than in Theorem~\ref{teo:A}.
\begin{athm}
\label{teo:B}
In the previous situation, for every $1 \leq i < j \leq r$ and every $1\leq k \leq r$, $(b_i\ast b_j) \ast b_k = (b_j \ast b_i) \ast b_k$.
\end{athm}

\begin{nota}
Let $B$ be a non-trivial $r$-generated left nilpotent brace of class~$2$. From Theorem~\ref{teo:B}, we conclude that either $B$ is right nilpotent of class~$2$ so that $r+1 \leq r(B) \leq r + r^2$, or $B$ is right nilpotent of class~$3$ so that $r+2 \leq r(B) \leq r + r^2\left(2 + \frac{r-1}{2}\right)$.
\end{nota}

The previous upper bound is effective as we describe a free object in the category of $r$-generated left nilpotent braces of class at most~$2$.


\begin{athm}
\label{teo:freergen}
Let $X$ be the alphabet given by the disjoint union of 
\[ \begin{array}{rl}
X_1 = \{\mathbf{x}_i \mid 1 \leq i \leq r\},& X_2 = \{\mathbf{x}_{ij} \mid 1 \leq i,j \leq r\},\\
X_{3,1} = \{\mathbf{x}_{iij} \mid 1 \leq i,j \leq r\},& X_{3,2} = \{\mathbf{x}_{ijk} \mid 1 \leq i < j \leq r, 1 \leq k \leq r, \}
\end{array}\]
and take $D$ the free abelian group over~$X$, so that every $\mathbf{d} \in D$ can be written as $\mathbf{d} = \sum_{\alpha} d_\alpha \mathbf{x}_{\alpha}$, where $d_\alpha \in \mathbb{Z}$ for every index $\alpha$ of an element $x_\alpha \in X$. Consider in $D$ the usual addition, given by $\mathbf{d} + \mathbf{e} = \mathbf{f}$, where $f_\alpha = d_\alpha + e_\alpha$ for every $\alpha$, and the product given by $\mathbf{d}\cdot\mathbf{e} = \mathbf{f}$, where
\[
\begin{array}{ll}
f_i = d_i + e_i, & \text{for each $1 \leq i \leq r$;}\\
f_{ij} = d_{ij} + e_{ij} + d_ie_j, & \text{for each $1 \leq i,j\leq r$;}\\ 
f_{iij} = d_{iij} + e_{iij} + \left(d_{ii} - \frac{d_i(d_i-1)}{2}\right)e_j, & \text{for each $1 \leq i,j\leq r$;}\\
f_{ijk} = d_{ijk}\!+\! e_{ijk}\! +\!(d_{ij} - d_id_j + d_{ji})e_k, & \text{for each $1 \leq i < j \leq r$, $1 \leq k \leq r$.}
\end{array}\]
Then, it follows that $(D,+,\cdot)$ is a brace such that $D^3 = 0$, and is $r$-generated as $D = \langle \mathbf{x}_1, \ldots, \mathbf{x}_r\rangle$. Moreover, given $C = \langle c_1, \ldots, c_r \rangle$ an $r$-generated brace for some $c_1, \ldots, c_r \in C$, and with $C^3 = 0$, there exists a unique homomorphism $\varphi\colon D \rightarrow C$ such that $\varphi(\mathbf{x}_i) = c_i$ for every $1\leq i \leq r$, and $\varphi(D) = \langle c_1, \ldots, c_r\rangle$.
\end{athm}

The one-generated case is studied in~\cite{KurdachenkoSubbotin24}:  there is a description of a $1$-generated brace $B$ with~\mbox{$B^3 = 0$} (cf.~\cite[Theorem~A]{KurdachenkoSubbotin24}); and a description of a free object in the category of $1$-generated left nilpotent braces of class at most~$2$ (cf.~\cite[Theorems~B$_1$ and B$_2$]{KurdachenkoSubbotin24}). These results are a direct consequence of the previous Theorems~\ref{teo:A} and~\ref{teo:freergen}, respectively.

\section{Proof of Theorem~\ref{teo:A}}

Throughout this section, we use the convention that star products come before products, written by juxtaposition, which in turn come before sums. Moreover, we use the following well-known properies of the star product:
\begin{align}
(ab) \ast c & = a \ast (b \ast c) + b \ast c + a \ast c; \nonumber \\
a \ast (b+c) & = a \ast b + a \ast c \label{dist_sum}\\
a + b &= ab - a \ast b
\end{align}
for every $a,b,c\in B$. Concretely, if $B^3 = 0$ then, for every $a,b,c\in B$
\begin{align}
(ab) \ast c & = a \ast c + b \ast c; \label{dist_prod}\\
a^{-1}\ast c & = -a \ast c; \label{inv_ast_-}\\
ac & = a+c, \ \text{if $c \in B^2$; in particular, $c^m = mc$, for every $m\in \mathbb{Z}$;} \label{B2inFix}\\
a+b & = ab(a \ast b)^{-1}. \label{a+b=abast}
\end{align}

We need the following lemma.

\begin{lem}
\label{lem:a^m=ma+a*a}
Let $B$ be a brace with $B^3 = 0$. Then, for every $a\in B$ and every $m \in \mathbb{Z}$, it holds
\[ a^m = ma + \frac{m(m-1)}{2}a \ast a.\]
Consequently, $ma = a^m(a \ast a)^s$ where $s = -\dfrac{m(m-1)}{2}$.
\end{lem}

\begin{proof}
Write $b = a \ast a$.  Assume by induction that $a^m =  ma + \frac{m(m-1)}{2}b$ for some $1\leq m$. The distributivity property of braces yields
\begin{align*}
a^{m+1} & = a(a^m) = a\left(ma + \frac{m(m-1)}{2}b\right)  \\
& =  ma^2 + \frac{m(m-1)}{2}ab - \left(m-1 + \frac{m(m-1)}{2}\right)a
\end{align*}
Since $a^2 = 2a + a \ast a = 2a + b$ and $ab = a + b$ by~\eqref{B2inFix}, we obtain that
\begin{align*}
a^{m+1} & = 2ma + mb + \frac{m(m-1)}{2}a + \frac{m(m-1)}{2}b - \left(m-1 + \frac{m(m-1)}{2}\right)a\\
& = (m+1)a + \left(m + \frac{m(m-1)}{2}\right)b = (m+1)a + \frac{(m+1)m}{2}b.
\end{align*}
Therefore, the equality holds for every positive integer $m$.

Now, observe that $a(a+a^{-1}) = a^2 -a = a + a \ast a = a(a\ast a)$. Therefore, it follows that $b = a \ast a = a + a^{-1} = a^{-1}\ast a^{-1}$. By the previous paragraph, we have that for every positive integer $m$
\begin{align*}
a^{-m} & = (a^{-1})^m  = ma^{-1} + \frac{m(m-1)}{2}a^{-1}\ast a^{-1} =  m(b - a) + \frac{m(m-1)}{2}b\\
& = (-m)a + \frac{m(m+1)}{2}b = (-m)a + \frac{(-m)(-m-1)}{2}b,
\end{align*}
and therefore, the property holds for every $m \in \mathbb{Z}$.

Using~\eqref{B2inFix}, we also conclude that $ma = a^m(a \ast a)^s$,  with $s = -\frac{m(m-1)}{2}$, for every $m\in \mathbb{Z}$.
\end{proof}

\begin{cor}
\label{cor:(ma)*(nc)}
Let $B$ be a brace with $B^3 = 0$. Then, for every integers $m,n$ and every $a,c\in B$ it holds
\begin{align*}
(ma) \ast (nc) = nm(a \ast c) - \frac{nm(m-1)}{2}(a\ast a) \ast c
\end{align*}
\end{cor}

\begin{proof}
By~\eqref{dist_sum}, it suffices to prove the corollary for $n = 1$. Then, the corollary follows after applying Lemma~\ref{lem:a^m=ma+a*a}, and the properites~\eqref{dist_prod} and~\eqref{inv_ast_-}.
\end{proof}

\begin{proof}[Proof of Theorem~\ref{teo:A}]
Call $S = L_1 + L_2 + L_3$ with
\begin{itemize}
\item $L_1 = \left\{\sum_{i=1}^r z_ib_i\mid z_i \in \mathbb{Z}, 1 \leq i \leq r\right\}$.
\item $L_2 = \left\{\sum_{i=1}^r z_{ij}b_{ij}\mid z_{ij} \in \mathbb{Z}, 1 \leq i,j \leq r\right\}\subseteq B^{2}$, where $b_{ij} = b_i \ast b_j$, for every $1\leq i, j\leq r$.
\item $L_3 = \left\{\sum_{i=1}^r z_{ijk}b_{ijk}\mid z_{ijk} \in \mathbb{Z}, 1 \leq i,j,k \leq r\right\} \subseteq B^{(3)}$, where $b_{ijk} = (b_i \ast b_j) \ast b_k$, for every $1\leq i, j, k\leq r$.
\end{itemize}
Clearly, $S$ is a subgroup of $(B,+)$ contained in the finitely generated brace $B = \langle b_1, \ldots, b_r \rangle$. Let us prove that $S$ is a subbrace of $B$. It suffices to show that $(S,\cdot)$ is a subgroup of $(B,\cdot)$.

Firstly, we claim that $b \ast b_k \in L_3$, for every $b \in L_2$ and every $1 \leq k \leq r$. If $b = \sum_{i=1}^r z_{ij}b_{ij}$ for some $z_{ij}\in \mathbb Z$,  $1 \leq i,j\leq r$, then we can write $b = \prod_{1\leq i,j\leq r} b_{ij}^{z_{ij}}$, as $L_2 \subseteq B^2$. Thus, by~\eqref{dist_prod},
\begin{align*}
b \ast b_k & = \left(\prod_{1\leq i,j\leq r} b_{ij}^{z_{ij}}\right) \ast b_k = \sum_{1\leq i,j\leq r} (z_{i,j}b_{i,j}) \ast b_k.
\end{align*}
Applying Corollary~\ref{cor:(ma)*(nc)}, 
\[ (z_{i,j}b_{i,j})\ast b_k = z_{i,j}(b_{i,j}\ast b_k) + \frac{z_{i,j}(z_{i,j}-1)}{2}(b_{i,j}\ast b_{i,j}) \ast b_k = z_{i,j}b_{ijk}\in L_3,\]
as $(b_{i,j}\ast b_{i,j}) \ast b_k \in B^{(4)} = 0$ by Theorem~\ref{teo:B3-B4}. Hence, $b \ast b_k \in L_3$ as we claimed.

Now, let $x = x_1 + x_2 + x_3, y = y_1 + y_2 + y_3 \in S$, with $x_i, y_i\in L_i$ for each $1\leq i \leq 3$. Thus, $c = x_2 + x_3, d = y_2 + y_3 \in B^2$. Then, applying~\eqref{dist_prod} and~\eqref{dist_sum}, it follows that
\begin{align}
x \ast y & = (x_1 + c) \ast (y_1 + d) = (x_1c) \ast (y_1 + d) = x_1 \ast y_1 + c \ast y_1, \label{xasty}
\end{align}
as $d\in B^2$. By the previous claim, $x_2 \ast y_1 \in L_3$, and $x_3 \ast y_1 \in B^{(4)} = 0$ by Theorem~\ref{teo:B3-B4}. Thus, by~\eqref{B2inFix} and~\eqref{dist_prod},
\[ c \ast  y_1 = (x_2 + x_3) \ast y_1 = (x_2x_3) \ast y_1 = x_2 \ast y_1\in L_3.\] 
It remains to show that $x_1\ast y_1 \in L_2 + L_3$.

Let $x_1 = \alpha_1b_1 + \ldots + \alpha_r b_r$, for some $\alpha_i \in \mathbb{Z}$ with $1\leq i \leq r$.  Clearly, by Corollary~\ref{cor:(ma)*(nc)}, it holds that
\begin{align*}
 (\alpha_i b_i) \ast b_j & = \alpha_i b_{ij} - \frac{\alpha_i(\alpha_i-1)}{2}b_{iij} \in L_2 + L_3
\end{align*}
for every $1\leq i,j\leq r$. Write $h =\alpha_1b_1 + \ldots + \alpha_{r-1}b_{r-1}$ and assume that $h \ast b_j \in L_2 + L_3$ for every $1\leq j \leq r$. By~\eqref{a+b=abast}, we have that $x_1 = h(\alpha_rb_r)(h\ast (\alpha_rb_r))^{-1}$. Thus, by~\eqref{dist_prod}, it follows that
\begin{align*}
x_1 \ast b_j & = \left(h(\alpha_rb_r)(h\ast (\alpha_rb_r))^{-1}\right)\! \ast b_j = h \ast b_j + (\alpha_rb_r)\! \ast b_j - (h \ast (\alpha_rb_r))\! \ast b_j.
\end{align*}
Since $h \ast (\alpha_rb_r) \in L_2 + L_3 \subseteq B^2$, we can write $h \ast (\alpha_rb_r) = l_2 + l_3 = l_2l_3$, with $l_2 \in L_2$ and $l_3 \in L_3$. Thus, 
\[ (h \ast (\alpha_rb_r)) \ast b_j = (l_2l_3) \ast b_j = l_2 \ast b_j + l_3 \ast b_j = l_2 \ast b_j \in L_3\]
as $l_3 \ast b_j \in B^{(4)} = 0$ by Theorem~\ref{teo:B3-B4}. Thus, $x_1 \ast b_j \in L_2 + L_3$ for every $1 \leq j \leq r$, and therefore, \eqref{dist_sum} yields  $x_1 \ast y_1 \in L_2 + L_3$.

Hence, we can conclude that $x \ast y = xy -x - y \in L_2 + L_3$. Thus, $xy \in S$.

Finally, it remains to show that for every $x \in S$, $x^{-1}\in S$. Let $x = x_1 + x_2 + x_3 \in S$. Call $c = x_2 + x_3 \in B^2$, and take $y = -x_1 + d$, where $d = -c + x_1 \ast x_1 + c \ast x_1 \in B^2$. Equation~\eqref{xasty} shows that
\[ x \ast y = x_1 \ast (-x_1) + c \ast (-x_1) = -(x_1 \ast x_1) -(c \ast x_1) = -c -d\]
On the other hand, 
\[ x \ast y = xy - x - y = xy -x_1 -c -(-x_1 + d) = xy -c -d.\]
Thus, $xy = 0$ and $y = x^{-1}$. Since we have seen that $x_1 \ast x_1, c \ast x_1 \in L_2 + L_3$, we have that $y = x^{-1}\in S$ as desired.
\end{proof}

\section{Proof of Theorem~\ref{teo:B}}
\label{teoB}
Throughout this section, assume that $B = \langle b_1, \ldots, b_r \rangle$, with $b_1, \ldots, b_r\in B$, is an $r$-generated left nilpotent brace of class at most~$2$. Recall that $b_{ij} = b_i \ast b_j$ for every $1\leq i,j\leq r$, and $b_{ijk} = (b_i \ast b_j) \ast b_k$ for every $1\leq i,j,k\leq r$. By Theorem~\ref{teo:A}, for an arbitrary $x \in B$, denote
\[ x = \sum_{1\leq i \leq r} x_i b_i + \sum_{1\leq i,j\leq r} x_{ij}b_{ij} + \sum_{1\leq i,j,k\leq r} x_{ijk}b_{ijk}\]
with $x_i, x_{jk}, x_{lmn} \in \mathbb{Z}$, for every $1\leq i,j,k,l,m,n\leq 2$.

We need the following lemmas.

\begin{lem}
\label{lem:dist_ni_bk}
Let $n_1, \ldots, n_m, z \in \mathbb{Z}$ with $1 \leq m \leq r$. Then, for every $1\leq k \leq r$,
\begin{align*}
\left(\sum_{1\leq i \leq m} n_ib_i\right) \ast(zb_k) & = \sum_{1 \leq i \leq m}\left(n_izb_{ik} - \frac{n_i(n_i-1)z}{2}b_{iik}\right) - \sum_{1\leq i < j \leq m} n_in_jzb_{ijk}
\end{align*}
\end{lem}

\begin{proof}
The case $m = 1$ is given by Corollary~\ref{cor:(ma)*(nc)}. Assume true the formula for some $1 \leq m-1 < r$. Applying~\eqref{a+b=abast} and~\eqref{dist_prod}, we see that
{\small
\begin{align*}
\left(\sum_{1\leq i \leq m} n_ib_i\right) \ast(zb_k) & = \left(\sum_{1\leq i \leq m-1} n_ib_i\right) \ast(zb_k) + (n_mb_m) \ast (zb_k) \ - \\
& - \left(\left(\sum_{1\leq i \leq m-1} n_ib_i\right) \ast (n_mb_m)\right) \ast (zb_k)
\end{align*}
}
By induction hypothesis, it holds that
{\small
\begin{align}
\!\left(\sum_{1\leq i \leq m-1} n_ib_i\right)\! \ast\!(zb_k) =\! \sum_{1 \leq i \leq m-1}\!\left(\!n_izb_{ik} - \frac{n_i(n_i-1)z}{2}b_{iik}\!\right)\! -\! \sum_{1\leq i < j \leq m-1} n_in_jzb_{ijk} \label{eq:hypind}
\end{align}
}
Moreover, Corollary~\ref{cor:(ma)*(nc)} yields
\begin{equation}
(n_mb_m) \ast (zb_k) = n_mzb_{mk} - \frac{n_m(n_m-1)z}{2}b_{mmk} \label{eq:(nmbm)*(zbk)}
\end{equation}
Once again, by induction hypothesis, we have that
\[ \left(\sum_{1\leq i \leq m-1} n_ib_i\right) \ast (n_mb_m) = \sum_{1 \leq i \leq m-1} n_in_mb_{im} + d\]
for some $d \in B^{(3)}$. Thus, by~\eqref{B2inFix} and~\eqref{dist_prod},
\begin{align}
\left(\left(\sum_{1\leq i \leq m-1} n_ib_i\right) \ast (n_mb_m)\right) \ast (zb_k) & = \left(\sum_{1 \leq i \leq m-1} n_in_mb_{im} + d\right) \ast (zb_k)  \nonumber \\
& = \sum_{1\leq i \leq m-1} n_in_mzb_{imk} \label{eq:signe-}
\end{align}
as $d \ast(zb_k) = 0$ because $B^{(4)} = 0$ by Theorem~\ref{teo:B3-B4}. Hence, the formula for $m$ holds after considering the sum of equations: \eqref{eq:hypind}+\eqref{eq:(nmbm)*(zbk)}-\eqref{eq:signe-}.
\end{proof}

\begin{lem}
\label{lem:producte_2-gen}
For every $x, y \in B$, if we write $z = xy$ then it holds that
\[
\begin{array}{ll}
z_i = x_i + y_i, & \text{for each $1 \leq i \leq r$;}\\
z_{ij} = x_{ij} + y_{ij} + x_iy_j, & \text{for each $1 \leq i,j\leq r$;}\\ 
z_{iij} = x_{iij} + y_{iij} + (x_{ii} - \frac{x_i(x_i-1)}{2})y_j, & \text{for each $1 \leq i,j\leq r$;}\\
z_{ijk} = x_{ijk}\!+\! y_{ijk} + (x_{ij} - x_ix_j)y_k, & \text{for each $1 \leq i < j \leq r$, $1 \leq k \leq r$;}\\
z_{ijk} = x_{ijk}\!+\! y_{ijk} + x_{ij}y_k, & \text{for each $1 \leq j < i \leq r$, $1 \leq k \leq r$.}
\end{array}\]
\end{lem}

\begin{proof}
Set
\[ c = \sum_{1\leq i,j\leq r} x_{ij}b_{ij} \in B^2,\ d = \sum_{1\leq i,j,k\leq r} x_{ijk}b_{ijk} \in B^{(3)}.\]
Equation~\eqref{xasty} in the proof of Theorem~\ref{teo:A} yields
\[ x \ast y = \left(\sum_{1\leq i \leq r} x_ib_i\right) \ast \left(\sum_{1\leq i \leq r} y_ib_i\right) + (c+d) \ast \left(\sum_{1\leq i \leq r} y_ib_i\right)\]
Applying Lemma~\ref{lem:dist_ni_bk}, for every $1 \leq k \leq r$ it holds that
{\small
\begin{align*}
\left(\sum_{1\leq i \leq m} x_ib_i\right) \ast(y_kb_k) & = \sum_{1 \leq i \leq m}\left(x_iy_kb_{ik} - \frac{x_i(x_i-1)y_k}{2}b_{iik}\right) - \sum_{1\leq i < j \leq m} x_ix_jy_kb_{ijk}
\end{align*}
}
On the other hand, since $d \in B^{(3)}$ and $B^{(4)} = 0$, it holds that $(c+d) \ast \left(\sum_{1\leq i \leq r} y_ib_i\right) = c \ast \left(\sum_{1\leq i \leq r} y_ib_i\right)$. Then, by~\eqref{B2inFix}, \eqref{dist_prod} and~\eqref{dist_sum}, we see that
\begin{align*}
c \ast (y_kb_k) & =  \sum_{1\leq i,j\leq r} x_{ij}y_kb_{ijk}, \quad \text{for every $1\leq k \leq r$.}
\end{align*}
If we call $u = x \ast y$, we have seen that $u_1 = \ldots = u_r = 0$, and from the above we get that
\[
\begin{array}{ll}
u_{ij} = x_iy_j, & \text{for every $1\leq i, j \leq r$;}\\
u_{iij} = \left(x_{ii} - \frac{x_i(x_i-1)}{2}\right)y_j, & \text{for every $1\leq i, j \leq r$;}\\
u_{ijk} = (x_{ij} -x_ix_j)y_k, & \text{for every $1\leq i < j \leq r$, $1\leq k \leq r$;}\\
u_{ijk} = x_{ij}y_k, & \text{for every $1\leq j < i \leq r$, $1\leq k \leq r$.}
\end{array}
\]
Hence, the lemma holds as $xy = x + y + x\ast y$.
\end{proof}

\begin{proof}[Proof of Theorem~\ref{teo:B}]
We show that the associative law in $(B,\cdot)$ implies $(b_i\ast b_j) \ast b_k = (b_j \ast b_i) \ast b_k$ for every $1\leq i < j \leq r$ and every $1\leq k \leq r$.

Call $u = xy$, $v = yz$, $s = (xy)z = uz$ and $t = x(yz) = xv$. According to Lemma~\ref{lem:producte_2-gen}, we write
\[
\begin{array}{ll}
u_i = x_i + y_i, & \text{for each $1 \leq i \leq r$;}\\
u_{ij} = x_{ij} + y_{ij} + x_iy_j, & \text{for each $1 \leq i,j\leq r$;}\\ 
u_{iij} = x_{iij} + y_{iij} + (x_{ii} - \frac{x_i(x_i-1)}{2})y_j, & \text{for each $1 \leq i,j\leq r$;}\\
u_{ijk} = x_{ijk}\!+\! y_{ijk} + (x_{ij} - x_ix_j)y_k, & \text{for each $1 \leq i < j \leq r$, $1 \leq k \leq r$;}\\
u_{ijk} = x_{ijk}\!+\! y_{ijk} + x_{ij}y_k, & \text{for each $1 \leq j < i \leq r$, $1 \leq k \leq r$.}
\end{array}\]
\[
\begin{array}{ll}
v_i = y_i + z_i, & \text{for each $1 \leq i \leq r$;}\\
v_{ij} = y_{ij} + z_{ij} + y_iz_j, & \text{for each $1 \leq i,j\leq r$;}\\ 
v_{iij} = y_{iij} + z_{iij} + (y_{ii} - \frac{y_i(y_i-1)}{2})z_j, & \text{for each $1 \leq i,j\leq r$;}\\
v_{ijk} = y_{ijk}\!+\! z_{ijk} + (y_{ij} - y_iy_j)z_k, & \text{for each $1 \leq i < j \leq r$, $1 \leq k \leq r$;}\\
v_{ijk} = y_{ijk}\!+\! z_{ijk} + y_{ij}z_k, & \text{for each $1 \leq j < i \leq r$, $1 \leq k \leq r$.}
\end{array}\]
\[
\begin{array}{ll}
s_i = u_i + y_i, & \text{for each $1 \leq i \leq r$;}\\
s_{ij} = u_{ij} + z_{ij} + u_iz_j, & \text{for each $1 \leq i,j\leq r$;}\\ 
s_{iij} = u_{iij} + z_{iij} + (u_{ii} - \frac{u_i(u_i-1)}{2})z_j, & \text{for each $1 \leq i,j\leq r$;}\\
s_{ijk} = u_{ijk}\!+\! z_{ijk} + (u_{ij} - u_iu_j)z_k, & \text{for each $1 \leq i < j \leq r$, $1 \leq k \leq r$;}\\
s_{ijk} = u_{ijk}\!+\! z_{ijk} + u_{ij}z_k, & \text{for each $1 \leq j < i \leq r$, $1 \leq k \leq r$.}
\end{array}\]
and
\[
\begin{array}{ll}
t_i = x_i + v_i, & \text{for each $1 \leq i \leq r$;}\\
t_{ij} = x_{ij} + v_{ij} + x_iv_j, & \text{for each $1 \leq i,j\leq r$;}\\ 
t_{iij} = x_{iij} + v_{iij} + (x_{ii} - \frac{x_i(x_i-1)}{2})v_j, & \text{for each $1 \leq i,j\leq r$;}\\
t_{ijk} = x_{ijk}\!+\! v_{ijk} + (x_{ij} - x_ix_j)v_k, & \text{for each $1 \leq i < j \leq r$, $1 \leq k \leq r$;}\\
t_{ijk} = x_{ijk}\!+\! v_{ijk} + x_{ij}v_k, & \text{for each $1 \leq j < i \leq r$, $1 \leq k \leq r$.}
\end{array}\]
Clearly, $s_i = t_i$ holds for every $1\leq i \leq r$. It also holds that
\begin{align*}
s_{ij} & = u_{ij} + z_{ij} + u_iz_j = x_{ij} + y_{ij} + x_iy_j + z_{ij} + (x_i+y_i)z_j = \\
& =  x_{ij} + y_{ij} + z_{ij} + x_iy_j + x_iz_j + y_iz_j;\\
t_{ij} & = x_{ij} + v_{ij} + x_iv_j = x_{ij} + y_{ij} + z_{ij} + y_iz_j + x_i(y_j + z_j) = \\
& = x_{ij} + y_{ij} + z_{ij} + x_iy_j + x_iz_j + y_iz_j.
\end{align*}
Therefore, $s_{ij} = t_{ij}$ for every  $1 \leq i,j \leq r$.

\medskip

Now, we see that
{\small
\begin{align*}
s_{iij} & = u_{iij} + z_{iij} +\! \left(u_{ii} - \frac{u_i(u_i-1)}{2}\right)\!z_j = x_{iij} + y_{iij} +\! \left(x_{ii} - \frac{x_i(x_i-1)}{2}\right)\!y_j \\
& + z_{iij} +\! \left(x_{ii} + y_{ii} + x_iy_i - \frac{(x_i+y_i)(x_i + y_i - 1)}{2}\right)\!z_j = \\
& = x_{iij} + y_{iij} + z_{iij} + x_{ii}y_j - \frac{x_i(x_i-1)}{2}y_j + \\
& + \!\left(x_{ii} + y_{ii} + x_iy_j - \frac{x_i(x_i-1)}{2} - \frac{y_i(y_i-1)}{2} - x_iy_i\right)\!z_j\\
& = x_{iij} + y_{iij} + z_{iij} + x_{ii}y_j + x_{ii}z_j + y_{ii}z_j \\
& - \frac{x_i(x_i-1)}{2}y_j - \frac{x_i(x_i-1)}{2}z_j - \frac{y_i(y_i-1)}{2}z_j;\\
t_{iij} & = x_{iij} + v_{iij} +\! \left(x_{ii} - \frac{x_i(x_i-1)}{2}\right)\!v_j = x_{iij} + y_{iij} + z_{iij} \\
& + \! \left(y_{ii} - \frac{y_i(y_i-1)}{2}\right)\!z_j + \!
\left(x_{ii} - \frac{x_i(x_i-1)}{2}\right)\!(y_j + z_j).
\end{align*}
}
Therefore, $s_{iij} = t_{iij}$ for every $1\leq i,j\leq r$.

Now, for every $1\leq i < j \leq r$ and every $1\leq k \leq r$, it holds that
\small{
\begin{align*}
s_{ijk} & = u_{ijk} + z_{ijk} +\! (u_{ij} - u_iu_j)z_k = x_{ijk} + y_{ijk}\! +\! (x_{ij} - x_ix_j)y_k + z_{ijk} \\
& +\!\big(x_{ij} + y_{ij} + x_iy_j - (x_i+y_i)(x_j + y_j)\big)z_k \\
& = x_{ijk} + y_{ijk} + z_{ijk} + x_{ij}y_k + x_{ij}z_k + y_{ij}z_k \\
& - x_ix_jy_k -x_ix_jz_k - y_ix_jz_k - y_iy_jz_k
\end{align*}
}
and 
\small{
\begin{align*}
t_{ijk} & = x_{ijk} + v_{ijk} +\! (x_{ij} - x_ix_j)v_k \\
& = x_{ijk} + y_{ijk} + z_{ijk} +\! (y_{ij} - y_iy_j)z_k +\! (x_{ij} - x_ix_j)(y_k + z_k) \\
& = x_{ijk} + y_{ijk} + z_{ijk} + x_{ij}y_k + x_{ij}z_k + y_{ij}z_k - x_ix_jy_k - x_ix_jz_k - y_iy_jz_k.
\end{align*}
}
Thus, $s_{ijk} - t_{ijk} = x_jy_iz_k$. On the other hand, for every $1\leq j < i \leq r$ and every $1 \leq k \leq r$, it holds that
\small{
\begin{align*}
s_{ijk} & = u_{ijk} + z_{ijk} + u_{ij}z_k = x_{ijk} + y_{ijk}\! + x_{ij}y_k +z_{ijk} + (x_{ij} + y_{ij} + x_iy_j)z_k \\
& = x_{ijk} + y_{ijk} + z_{ijk} + x_{ij}y_k + x_{ij}z_k + y_{ij}z_k + x_iy_jz_k
\end{align*}
}
and
\small{
\begin{align*}
t_{ijk} & = x_{ijk} + v_{ijk} + x_{ij}v_k = x_{ijk} +
y_{ijk} + z_{ijk} + y_{ij}z_k + x_{jk}(y_k + z_k) \\
& = x_{ijk} + y_{ijk} + z_{ijk} + x_{ij}y_k + x_{ij}z_k + y_{ij}z_k
\end{align*}
}
Therefore, $s_{ijk} - t_{ijk} = x_iy_jz_k$. Hence, the associative law in the multiplicative group yields
\[ \sum_{1\leq i < j \leq r, 1\leq k \leq r} x_jy_iz_k(b_{ijk} - b_{jik}) 
 = 0\]
for every $x_j,y_i,z_k \in \mathbb{Z}$. Fix $1 \leq k \leq r$. If we take $x_j = y_i = 1$, $z_k = 1$, and $z_{k'} = 0$ for every $1\leq k\neq k' \leq r$,  it follows that $b_{ijk} = b_{jik}$.
\end{proof}

\begin{nota}
\label{nota:nou-producte2gen}
From Theorem~\ref{teo:B}, every element $x \in B$ can be written as
\[ x = \sum_{1\leq i \leq r} x_i b_i + \sum_{1\leq i,j\leq r} x_{ij}b_{ij} + \sum_{1\leq i<j\leq r, 1 \leq k \leq r} x_{ijk}b_{ijk}\]
Moreover, for every $x,y\in B$, if we write $z = xy$ then Lemma~\ref{lem:producte_2-gen} shows that
{\small
\[
\begin{array}{ll}
z_i = x_i + y_i, & \text{for each $1 \leq i \leq r$;}\\
z_{ij} = x_{ij} + y_{ij} + x_iy_j, & \text{for each $1 \leq i,j\leq r$;}\\ 
z_{iij} = x_{iij} + y_{iij} + \left(x_{ii} - \frac{x_i(x_i-1)}{2}\right)y_j, & \text{for each $1 \leq i,j\leq r$;}\\
z_{ijk} = x_{ijk}\!+\! y_{ijk} + (x_{ij} - x_ix_j + x_{ji})y_k, & \text{for each $1 \leq i < j \leq r$, $1 \leq k \leq r$.}
\end{array}\]
}
In particular, for every $x \in B$, if we write $z = x^{-1}$ then it holds that
{\small
\[
\begin{array}{ll}
z_i = -x_i, & \text{for each $1\leq i \leq r$;} \\
z_{ij} = x_ix_j -x_{ij}, & \text{for each $1\leq i, j \leq r$;} \\
z_{iij} = \left(x_{ii} - \frac{x_i(x_i-1)}{2}\right)x_j - x_{iij}, & \text{for each $1\leq i, j \leq r$;}\\
z_{ijk} = (x_{ij} - x_ix_j + x_{ji})x_k - x_{ijk}, & \text{for each $1\leq i < j \leq r$, $1\leq k \leq r$.}
\end{array}\]
}
\end{nota}

\section{Proof of Theorem~\ref{teo:freergen}}
\begin{proof}[Proof of Theorem~\ref{teo:freergen}]
The associativity of the product in $D$ follows by similar arguments as in the proof of Theorem~\ref{teo:B} (see also Remark~\ref{nota:nou-producte2gen}). Thus, $(D,\cdot)$ is a group, whose identity element is $\mathbf{0} \in D$. Moreover, if $\mathbf{d} \in D$ then $\mathbf{f} = \mathbf{d}^{-1}\in D$ is given by
{\small
\[
\begin{array}{ll}
f_i = -d_i, & \text{for each $1\leq i \leq r$;} \\
f_{ij} = d_id_j -d_{ij}, & \text{for each $1\leq i, j \leq r$;} \\
f_{iij} = \left(d_{ii} - \frac{d_i(d_i-1)}{2}\right)d_j - d_{iij}, & \text{for each $1\leq i, j \leq r$;}\\
f_{ijk} = (d_{ij} - d_id_j + d_{ji})d_k - d_{ijk}, & \text{for each $1\leq i < j \leq r$, $1\leq k \leq r$.}
\end{array}\]
}
Now, let us see that $(D,+,\cdot)$ satisfies the brace distributivity property. Take $\mathbf{x},\mathbf{y}, \mathbf{z} \in D$. Call $\mathbf{u} = \mathbf{xy}$, $\mathbf{v} = \mathbf{xz}$, and $\mathbf{s} = \mathbf{x(y+z)}$. Then, it follows that
{\small
\[
\begin{array}{l}
\begin{array}{ll}
s_i = x_i + y_i + z_i, & \text{for each $1\leq i \leq r$;}\\
s_{ij} = x_{ij} + y_{ij} + x_i(y_j+z_j), & \text{for each $1\leq i,j  \leq r$;} \\
s_{iij} = x_{iij} + y_{iij} + z_{iij} +\!(x_{ii} - \frac{x_i(x_i-1)}{2})(y_j+z_j), &  \text{for each $1\leq i,j  \leq r$;} \\
\end{array}\\
\begin{array}{ll}
s_{ijk} =  x_{ijk} + y_{ijk} + z_{ijk} + \! (x_{ij} - x_ix_j + x_{ji})(y_k + z_k), & \text{for each $1\leq i < j \leq r$,}\\
& 1\leq k \leq r.
\end{array}
\end{array}\]
}
On the other hand, we also have that
\[
\begin{array}{ll}
u_i = x_i + y_i, & \text{for each $1 \leq i \leq r$;}\\
u_{ij} = x_{ij} + y_{ij} + x_iy_j, & \text{for each $1 \leq i,j\leq r$;}\\ 
u_{iij} = x_{iij} + y_{iij} + \left(x_{ii} - \frac{x_i(x_i-1)}{2}\right)y_j, & \text{for each $1 \leq i,j\leq r$;}\\
u_{ijk} = x_{ijk}\!+\! y_{ijk}\! +\!(x_{ij} - x_ix_j + x_{ji})y_k, & \text{for each $1 \leq i < j \leq r$, $1 \leq k \leq r$.}
\end{array}\]
and analogously
\[
\begin{array}{ll}
v_i = x_i + z_i, & \text{for each $1 \leq i \leq r$;}\\
v_{ij} = x_{ij} + z_{ij} + x_iz_j, & \text{for each $1 \leq i,j\leq r$;}\\ 
v_{iij} = x_{iij} + z_{iij} + \left(x_{ii} - \frac{x_i(x_i-1)}{2}\right)z_j, & \text{for each $1 \leq i,j\leq r$;}\\
v_{ijk} = x_{ijk}\!+\! z_{ijk}\! +\!(x_{ij} - x_ix_j + x_{ji})z_k, & \text{for each $1 \leq i < j \leq r$, $1 \leq k \leq r$.}
\end{array}\]
If we call $\mathbf{t} = \mathbf{xy} - \mathbf{x} + \mathbf{xz} = \mathbf{u} - \mathbf{x} + \mathbf{v}$, we obtain
{\small
\begin{align*}
t_i & = u_i - x_i + v_i = x_i + y_i + z_i = s_i,\quad  \text{for each $1 \leq i \leq r$;}
\end{align*}
\begin{align*}
t_{ij} & = u_{ij} - x_{ij} + v_{ij} = x_{ij} + y_{ij} + x_iy_j - x_{ij} + x_{ij} + z_{ij} + x_iz_j \\
& = x_{ij} + y_{ij} + z_{ij} +x_iy_j + x_iz_j = s_{ij},\quad  \text{for each $1 \leq i,j\leq r$;}
\end{align*}
\begin{align*}
t_{iij} & = u_{iij} - x_{iij} + v_{iij} = x_{iij} + y_{iij} +\!\text{$\left(x_{ii} - \frac{x_i(x_i-1)}{2}\right)$}y_j - x_{iij} + x_{iij} + z_{iij}\,+\\
& \, + \!\text{$\left(x_{ii} - \frac{x_i(x_i-1)}{2}\right)$}z_j = x_{iij} + y_{iij} + z_{iij} +\!\text{$\left(x_{ii} - \frac{x_i(x_i-1)}{2}\right)$}(y_j+z_j) \\
& = s_{iij},\quad \text{for each $1 \leq i,j\leq r$;}
\end{align*}
\begin{align*}
t_{ijk} & = u_{ijk} - x_{ijk} + v_{ijk} = x_{ijk} + y_{ijk} +\!(x_{ij} - x_ix_j + x_{ji})y_k - x_{ijk} + x_{ijk} \,+\\
&\, +\, z_{ijk} +\! (x_{ij} - x_ix_j + x_{ji})z_k = x_{ijk} + y_{ijk} + z_{ijk} + \! (x_{ij} - x_ix_j + x_{ji})(y_k + z_k)\\
& = s_{ijk}, \quad \text{for each $1\leq i < j \leq r$, $1\leq k \leq r$.}\\
\end{align*}}
Hence, $(D,+,\cdot)$ satisfies the brace distributivity property, and therefore, it is a brace. 

Clearly, $D$ is finitely generated as it is finitely additively generated by the elements of $X$. For every $\mathbf{x}_\alpha \in X$, and every index $\alpha'$ of the alphabet $X$, we denote $(x_\alpha)_{\alpha'}$ the $\alpha'$ index of $\mathbf{x}_\alpha$. Clearly, it holds that $(x_\alpha)_{\alpha'} = 1$ if $\alpha = \alpha'$, and $(x_\alpha)_{\alpha'} = 0$ otherwise.

From the product formula, it is a routine to check
\[
\mathbf{x}_i \ast \mathbf{x}_j  = -\mathbf{x}_i + \mathbf{x}_i\mathbf{x}_j - \mathbf{x}_j = \mathbf{x}_{ij},\quad \text{for each $1 \leq i,j\leq r$.}\]
Fix $1\leq i,j \leq r$ and take $\mathbf{u} = (\mathbf{x}_i \ast \mathbf{x}_i) \ast \mathbf{x}_j = \mathbf{x}_{ii} \ast \mathbf{x}_j  = -\mathbf{x}_{ii} + \mathbf{x}_{ii}\mathbf{x}_j - \mathbf{x}_j$ and $\mathbf{v} = \mathbf{x}_{ii}\mathbf{x}_j$. Observe that
\begin{align*}
v_{ii} & = (x_{ii})_{ii} + (x_j)_{ii} + (x_{ii})_i(x_j)_i = 1 + 0 + 0 = 1\\
v_{j} & = (x_{ii})_j + (x_j)_j  = 0 + 1 = 1\\
v_{iij} & = (x_{ii})_{iij}\! +\! (x_j)_{iij}\! +\! \left((x_{ii})_{ii}\! +\! \frac{(x_{ii})_i\big((x_{ii})_i -1\big)}{2}\right)\!(x_j)_j\! = 0 + 0 + 1 + 0 = 1
\end{align*}
Similarly, we can check that if $\alpha \notin \{ii, j, iij\}$, then $v_\alpha = 0$. Thus, it follows that
\begin{align*}
u_{ii} & = -(x_{ii})_{ii} + v_{ii} - (x_j)_{ii} = -1 + 1 + 0 = 0,\\
u_{j} & = -(x_{ii})_j + v_j - (x_j)_j = 0 + 1 - 1 = 0,\\
u_{iij} & = -(x_{ii})_{iij} + v_{iij} - (x_j)_{iij} = 0 + 1 + 0 = 1,
\end{align*}
and $u_\alpha = -(x_{ii})_\alpha + v_\alpha - (x_j)_\alpha = 0$ for every $\alpha \notin \{ii, j, iij\}$. 

Therefore, $(\mathbf{x}_i \ast \mathbf{x}_i) \ast \mathbf{x}_j =\mathbf{x}_{iij}$ for every $1\leq i,j \leq r$.

Fix $1 \leq i < j \leq r$ and $1 \leq k \leq r$, and take $\mathbf{u} = (\mathbf{x}_i \ast \mathbf{x}_j) \ast \mathbf{x}_k = \mathbf{x}_{ij} \ast \mathbf{x}_k  = -\mathbf{x}_{ij} + \mathbf{x}_{ij}\mathbf{x}_k - \mathbf{x}_k$ and $\mathbf{v} = \mathbf{x}_{ij}\mathbf{x}_k$. Observe that
\begin{align*}
v_{ij} & = (x_{ij})_{ij} + (x_k)_{ij} + (x_{ij})_i(x_k)_j = 1 + 0 + 0 = 1\\
v_{k} & = (x_{ij})_k + (x_k)_k  = 0 + 1 = 1\\
v_{ijk} & = (x_{ij})_{ijk}\! +\! (x_k)_{ijk}\! +\! \big((x_{ij})_{ij}\! +\! (x_{ij})_i(x_{ij})_j \!+\! (x_{ij})_{ji}\big)(x_k)_k\! = 0 + 0 + 1 = 1
\end{align*}
Similarly, we can check that if $\alpha \notin \{ij, k, ijk\}$, then $v_\alpha = 0$. Thus, it follows that
\begin{align*}
u_{ij} & = -(x_{ij})_{ij} + v_{ij} - (x_k)_{ij} = -1 + 1 + 0 = 0,\\
u_{k} & = -(x_{ij})_k + v_k - (x_k)_k = 0 + 1 - 1 = 0,\\
u_{ijk} & = -(x_{ij})_{ijk} + v_{ijk} - (x_k)_{ijk} = 0 + 1 + 0 = 1,
\end{align*}
and $u_\alpha = -(x_{ij})_\alpha + v_\alpha - (x_k)_\alpha = 0$ for every $\alpha \notin \{ij, k, ijk\}$.

Therefore, $(\mathbf{x}_i \ast \mathbf{x}_j) \ast \mathbf{x}_k =\mathbf{x}_{ijk}$ for every $1\leq i,j \leq r$, and every $1\leq k \leq r$.

Hence, we conclude that $D = \langle \mathbf{x}_1, \ldots, \mathbf{x}_r\rangle$ is $r$-generated.

\medskip

Now, take $\mathbf{d} \in D$ and $\mathbf{x}_\alpha \in D$. Call $\mathbf{u} = \mathbf{d}\ast \mathbf{x}_\alpha = -\mathbf{d} + \mathbf{dx}_\alpha - \mathbf{x}_\alpha$, and $\mathbf{v} = \mathbf{dx}_\alpha$. For every $1\leq i \leq r$, it holds that $v_i = d_i + (x_\alpha)_i$. Thus, it follows that $u_i = -d_i + d_i + (x_\alpha)_i - (x_\alpha)_i = 0$, for every $1\leq i \leq r$. Therefore, $D^2 = \langle \mathbf{x}_\alpha\mid \alpha \notin \{1,\ldots, r\}\rangle$. 

Take $\mathbf{d} \in D$ and $\mathbf{x}_\alpha \in D$ with $\alpha \notin \{1,\ldots, r\}$. As before, $\mathbf{u} = \mathbf{d}\ast \mathbf{x}_\alpha = -\mathbf{d} + \mathbf{dx}_\alpha - \mathbf{x}_\alpha$, and $\mathbf{v} = \mathbf{dx}_\alpha$. Since $\alpha \notin \{1,\ldots, r\}$, from the product formula it follows that $v_{\alpha'} = d_{\alpha'} + (x_\alpha)_{\alpha'}$ for every index $\alpha'$ of the alphabet $X$. Thus,
\[ u_{\alpha'} = -d_{\alpha'} +d_{\alpha'} + (x_\alpha)_{\alpha'} = 0, \text{for every index $\alpha'$ of $X$.}\]
Applying equation~\eqref{dist_sum}, this means that $\mathbf{d}\ast \mathbf{e} = 0$ for every $e \in D^2$. Hence, $D^3 = D \ast (D\ast D) = 0$, and therefore, $D^{(4)} = 0$ by Theorem~\ref{teo:B3-B4}.

\medskip

Finally, let $C$ be a brace with $C^3 = 0$ and such that $C = \langle c_1, \ldots, c_r\rangle$ is $r$-generated. By Theorem~\ref{teo:B}, every element $x \in C$ can be written as a sum
\[ x = \sum_{1\leq i \leq r} x_i c_i + \sum_{1\leq i,j\leq r} x_{ij}c_{ij} + \sum_{1\leq i<j\leq r, 1 \leq k \leq r} x_{ijk}c_{ijk}\]
where $c_{ij} = c_i \ast c_j$ and $c_{klm} = (c_k\ast c_l) \ast c_m$, for every $1\leq i,j\leq r$, and every $1\leq k < l \leq r$, $1\leq m\leq r$.
Therefore, the map $\varphi\colon D \rightarrow B$ sending each $d \in D$ with
\[ d = \sum_{1\leq i \leq r} d_i \mathbf{x}_i + \sum_{1\leq i,j\leq r} d_{ij}\mathbf{x}_{ij} + \sum_{1\leq i<j\leq r, 1 \leq k \leq r} d_{ijk}\mathbf{d}_{ijk}\]
to
\[ c = \sum_{1\leq i \leq r} d_i c_i + \sum_{1\leq i,j\leq r} d_{ij}c_{ij} + \sum_{1\leq i<j\leq r, 1 \leq k \leq r} d_{ijk}c_{ijk}\]
defines a brace epimorphism, with $\varphi(\mathbf{x}_i) = c_i$ for every $1\leq i \leq r$. Moreover, this is the unique possible homomorphism with this condition.
\end{proof}

\section*{Acknowledgements}
The first author is sponsored by the Natural Science Foundation of Shanghai (24ZR1422800) and the National Natural Science Foundation of China (12471018). The second and fourth authors are supported by the Conselleria d'Educació, Universitats i Ocupació, Generalitat Valenciana (grant: \mbox{CIAICO/2023/007}). The second author is also supported by the program "High-end Foreing Expert Program of China", Shanghai University. The third author is very grateful to the Conselleria d'Innovaci\'o, Universitats, Ci\`encia i
Societat Digital of the Generalitat (Valencian Community, Spain) and the Universitat de
Val\`encia for their financial support and grant to host researchers affected by the war in
Ukraine in research centres of the Valencian Community. The third author would also like to thank the Isaac Newton Institute for Mathematical Sciences, Cambridge, for support and hospitality during the Solidarity Supplementary Grant Program. This work was supported by EPSRC grant no EP/R014604/1. He is sincerely grateful to Agata Smoktunowicz.


\begin{thebibliography}{10}

\bibitem{BallesterEstebanFerraraPerezCTrombetti25-central}
A.~Ballester-Bolinches, R.~Esteban-Romero, M.~Ferrara, V.~P{\'e}rez-Calabuig,
  and M.~Trombetti.
\newblock Central nilpotency of left skew braces and solutions of the
  {Y}ang-{B}axter equation.
\newblock {\em Pac. J. Math.}, 335(1):1--32, 2025.

\bibitem{BallesterEstebanKurdachenkoPerezC25-onegenerated}
A.~Ballester-Bolinches, R.~Esteban-Romero, L.~A. Kurdachenko, and
  V.~P{\'e}rez-Calabuig.
\newblock From actions of an abelian group on itself to left braces.
\newblock {\em Math. Proc. Camb. Philos. Soc.}, 178(1):65--79, 2025.

\bibitem{BallesterKurdachenkoPerezC-arXiv-B3-B4}
A.~Ballester-Bolinches, L.~A. Kurdachenko, and V.~P{\'e}rez-Calabuig.
\newblock On left nilpotent skew braces of class 2.
\newblock {\em In arXiv:2505.07115}, 2025.

\bibitem{BonattoJedlicka23}
M.~Bonatto and P.~Jedli\v{c}ka.
\newblock Central nilpotency of skew braces.
\newblock {\em J. Algebra Appl.}, 22(12):2350255, 2023.



\bibitem{DixonKurdachenkoSubbotin25-preprint}
M.~R. Dixon, L.~A. Kurdachenko, and I.~Ya. Subbotin.
\newblock On the structure of some one-generator nilpotent brace.
\newblock {\em In arXiv: 2501.04567}, 2025.

\bibitem{JespersVanAntwerpenVendramin23}
E.~Jespers, A.~Van~Antwerpen, and L.~Vendramin.
\newblock Nilpotency of skew braces and multipermutation solutions of the
  {Y}ang–{B}axter equation.
\newblock {\em Commun. Contemp. Math.}, 25(09):2250064, 2023.

\bibitem{KurdachenkoSubbotin24}
L.A. Kurdachenko and I.~Ya. Subbotin.
\newblock On the structure of some one-generator braces.
\newblock {\em Proc. Edinb. Math. Soc.}, 67:566--576, 2024.

\bibitem{Rump07}
W.~Rump.
\newblock Braces, radical rings, and the quantum {Y}ang-{B}axter equation.
\newblock {\em J. Algebra}, 307:153--170, 2007.

\bibitem{Smoktunowicz18-tams}
A.~Smoktunowicz.
\newblock On {E}ngel groups, nilpotent groups, rings, braces and the
  {Y}ang-{B}axter equation.
\newblock {\em Trans.\ Amer.\ Math.\ Soc.}, 370(9):6535--6564, 2018.

\end{thebibliography}
\end{document}